\documentclass[a4paper]{amsart}

\usepackage{amssymb, enumitem}
\usepackage[all]{xy}
\usepackage{hyperref, aliascnt, graphicx}
\usepackage{subfigure}
\usepackage{pst-all}

\newtheorem{lma}{Lemma}[section]

\newaliascnt{thmCt}{lma}
\newtheorem{thm}[thmCt]{Theorem}
\aliascntresetthe{thmCt}

\newaliascnt{corCt}{lma}
\newtheorem{cor}[corCt]{Corollary}
\aliascntresetthe{corCt}

\newaliascnt{prpCt}{lma}
\newtheorem{prp}[prpCt]{Proposition}
\aliascntresetthe{prpCt}

\newtheorem*{thm*}{Theorem}
\newtheorem*{cor*}{Corollary}
\newtheorem*{prop*}{Proposition}

\theoremstyle{definition}

\newaliascnt{pgrCt}{lma}

\aliascntresetthe{pgrCt}

\newaliascnt{dfnCt}{lma}
\newtheorem{dfn}[dfnCt]{Definition}
\aliascntresetthe{dfnCt}

\newaliascnt{rmkCt}{lma}
\newtheorem{rmk}[rmkCt]{Remark}
\aliascntresetthe{rmkCt}

\newaliascnt{rmksCt}{lma}
\newtheorem{rmks}[rmksCt]{Remarks}
\aliascntresetthe{rmksCt}

\newaliascnt{exaCt}{lma}
\newtheorem{exa}[exaCt]{Example}
\aliascntresetthe{exaCt}

\newaliascnt{exasCt}{lma}
\newtheorem{exas}[exasCt]{Examples}
\aliascntresetthe{exasCt}

\newaliascnt{cnjCt}{lma}
\newtheorem{cnj}[cnjCt]{Conjecture}
\aliascntresetthe{cnjCt}

\DeclareMathOperator{\Hom}{Hom}
\newcommand{\ZZ}{{\mathbb{Z}}}
\newcommand{\NN}{{\mathbb{N}}}
\newcommand{\CC}{{\mathbb{C}}}
\newcommand{\RR}{{\mathbb{R}}}
\newcommand{\TT}{{\mathbb{T}}}
\newcommand{\id}{{\mathrm{id}}}
\newcommand{\ca}{$C^*$-algebra}
\newcommand{\stHom}{${}^*$-homomorphism}
\newcommand{\freeVar}{\_\,}
\newcommand{\CatC}{\ensuremath{\mathbf{C}}}
\newcommand{\CatGp}{\ensuremath{\mathbf{Gp}}}
\newcommand{\CatSet}{\ensuremath{\mathbf{Set}}}
\newcommand{\CatMon}{\ensuremath{\mathbf{Mon}}}
\newcommand{\CatCa}{\ensuremath{\mathbf{C}^*}}
\newcommand{\CatCaUnit}{\ensuremath{\mathbf{C}^*_1}}
\newcommand{\CatCaSepUnit}{\ensuremath{\mathbf{SC}^*_1}}
\newcommand{\CatCaCommSepUnit}{\ensuremath{\mathbf{AbSC}^*_1}}
\newcommand{\CatCpctMetr}{\ensuremath{\mathbf{CMetr}}}

\title{Semiprojectivity and semiinjectivity in different categories}
\date{\today}

\author{Hannes Thiel}
\address{Hannes Thiel
Mathematisches Institut, Fachbereich Mathematik und Informatik der
Universit\"at M\"unster, Einsteinstrasse 62, 48149 M\"unster, Germany.}
\email{hannes.thiel@uni-muenster.de}
\urladdr{www.math.uni-muenster.de/u/hannes.thiel/}

\thanks{The author was partially supported by the Deutsche Forschungsgemeinschaft (SFB 878).}


\subjclass[2010]%
{Primary
18A05; 
Secondary
06B35, 
06F05, 
18A20, 
20E05, 
46L05, 
54C55, 
55M15. 
}

\begin{document}

\begin{abstract}
Projectivity and injectivity are fundamental notions in category theory.
We consider natural weakenings termed semi\-projectivity and semiinjectivity, and study these concepts in different categories.

For example, in the category of metric spaces, (semi)injective objects are precisely the absolute (neighborhood) retracts.
We show that the trivial group is the only semiinjective group, while every free product of a finitely presented group and a free group is semiprojective.

To a compact, metric space $X$ we associate the commutative \ca{} $C(X)$.
This association is contravariant, whence semiinjectivity of $X$ is related to semiprojectivity of $C(X)$.
Together with Adam S{\o}rensen, we showed that $C(X)$ is semiprojective in the category of all \ca{s} if and only if $X$ is an absolute neighborhood retract with $\dim(X)\leq 1$.
\end{abstract}

\maketitle

\section{Introduction}

While being fundamental in category theory, the concepts of projectivity and injectivity are often very restrictive.
It is therefore natural to consider weaker versions of these notions.

For example, injective objects in the category of metric spaces and continuous maps are precisely the absolute retracts introduced by Borsuk in 1931.
He also defined a generalization, called absolute neighborhood retracts;
see \autoref{dfn:ANR}.
Therefore, being an absolute neighborhood retract is a weak form of injectivity in the category of metric spaces.
While not many spaces are absolute retracts, numerous naturally occurring spaces are absolute neighborhood retracts, including topological manifolds, polyhedra and CW-complexes.

Given a compact, metric space $X$, we associate the algebra $C(X)$ of continuous complex-valued functions on $X$.
This is a \ca{} with the supremum norm and pointwise operations.
The category \CatCpctMetr{} of compact, metric spaces, and the category \CatCaCommSepUnit{} of abelian, unital, separable \ca{s} are (contravariantly) equivalent.
We therefore think of \ca{s} as `noncommutative topological spaces'.

Using the contravariant correspondence between \CatCpctMetr{} and \CatCaCommSepUnit{}, injectivity of a compact, metric space $X$ corresponds precisely to projectivity of $C(X)$ in the category \CatCaCommSepUnit{}.
However, projectivity within the category \CatCaUnit{} of unital \ca{s} and unital \stHom{s} is more restrictive - since there are more lifting problems to be solved.

In 1985, Blackadar used the contravariant correspondence between \CatCpctMetr{} and \CatCaCommSepUnit{} to translate the concept of an absolute neighborhood retract to (noncommutative) \ca{s}.
He hence introduced a weak form of projectivity in the category of \ca{s}, called \emph{semiprojectivity};
see \cite[Definition~2.1]{Bla85NCShapeThy}, \cite[Definition~2.10]{Bla85ShapeThy}.

It is straightforward to generalize Blackadar's definition of semiprojectivity to general categories;
see \autoref{dfn:semiproj}.
The dual notion is called \emph{semiinjectivity}.
Semiinjective objects in the category of metric spaces are precisely the absolute neighborhood retracts.

We characterize semiprojectivity and semiinjectivity in the category of groups:
A group is semiprojective if and only if it is a retract of a free product of a finitely presented group and a free group;
see \autoref{prp:semiprojGp}.
On the other hand, only the trivial group is semiinjective;
see \autoref{prp:semiinjGp}.

One motivation to consider semiprojectivity and semiinjectivity is shape theory, which is a machinery to study an object by approximating it by better-behaved ones.
Absolute neighborhood retracts are the building blocks of shape theory of topological spaces.
Analogously, semiprojective \ca{s} are the building blocks of noncommutative shape theory.
In this context, it is of general interest to study semiprojective \ca{s}.

More specifically, semiprojectivity is a concept that is used at many different places in the theory of \ca{s}.
For instance, it is often used that a \ca{} that is `locally approximated' by a certain class of semiprojective \ca{s} is already isomorphic to an inductive limit of such \ca{s};
see \cite[Section~3]{Thi11arX:IndLimPj}.
Further, semiprojective \ca{s} are used to study and classify \ca{s} given as inductive limits or as crossed products of dynamical systems.
For example, Elliott's seminal classification of AF-algebras (inductive limits of finite-dimensional \ca{s}) by $K$-theory relies on the semiprojectivity of finite-dimensional \ca{s}.
Semiprojectivity also plays a crucial role in the analysis of the structure of crossed products by actions with the Rokhlin property in \cite{OsaPhi12CrProdRokhlin} and \cite{Gar17CrProdCpctRP}.

Blackadar asked to determine, in terms of $X$, when $C(X)$ is semiprojective among all \ca{s}.
It is easy to see that $X$ must be an absolute neighborhood retract, but is that sufficient?
Surprisingly, a dimensional restriction appears.
Together with Adam S{\o}rensen we showed in \cite{SoeThi12CharCommutSP} that $C(X)$ is semiprojective if and only if $X$ is an absolute neighborhood retract with $\dim(X)\leq 1$;
see \autoref{prp:CharNcANR}.
\\

This article is based on a talk presented at the conference `VI Coloquio Uruguayo
de Matem\'{a}tica', held during December 20 to 22, 2017, in Montevideo, Uruguay.

\section*{Acknowledgements}

I am thankful to Eusebio Gardella for his feedback on the first draft of this paper.

\section{Concrete categories, Monomorphisms, Epimorphisms}

Let \CatC{} be a category.
Given objects $X$ and $Y$ in \CatC, we use $\Hom_\CatC(X,Y)$ to denote the morphisms in \CatC{} from $X$ to $Y$.

Let $\varphi\colon X\to Y$ be a morphism in \CatC.
Then $\varphi$ is an \emph{epimorphism} (for short, $\varphi$ is \emph{epi}), denoted $\varphi\colon X\twoheadrightarrow Y$, if for every object $Z$ and morphisms $\psi_1,\psi_2\colon Y\to Z$ with $\psi_1\circ\varphi=\psi_2\circ\varphi$, we have $\psi_1=\psi_2$.
Dually, $\varphi$ is a \emph{monomorphism} (for short, $\varphi$ is \emph{mono}), denoted $\varphi\colon X\hookrightarrow Y$, if for every object $Z$ and morphisms $\psi_1,\psi_2\colon Z\to X$ with $\varphi\circ\psi_1=\varphi\circ\psi_2$, we have $\psi_1=\psi_2$.

\begin{exa}
Let \CatSet{} denote the category of sets and (ordinary) mappings.
A morphism in \CatSet{} is epi (mono) if and only if it is surjective (injective).
\end{exa}

Recall that \CatC{} is said to be \emph{locally small} if $\Hom_\CatC(X,Y)$ is a set for any objects $X$ and $Y$.
In this case, for each object $X$ we obtain a covariant hom functor $\Hom_\CatC(X,\freeVar)\colon\CatC\to\CatSet$ and a contravariant hom functor $\Hom_\CatC(\freeVar,X)\colon\CatC\to\CatSet$.

A \emph{concrete} category is a category $\CatC$ together with faithful functor $U\colon\CatC\to\CatSet$.
In this case, we think of an object $X$ in \CatC{} as a set (namely $U(X)$) with additional structure, and a morphism $\varphi\colon X\to Y$ is a mapping (namely $U(\varphi)\colon U(X)\to U(Y)$) that preserves the structure of the objects.
In a concrete category, we usually identify an object with its underlying set, and we identify a morphism with its underlying set mapping.

A faithful functor reflects epimorphisms and monomorphisms.
It follows that in a concrete category, every surjective (injective) morphism is epi (mono).
The converse need not hold;
see Examples~\ref{exa:Mon} and~\ref{exa:nonInjMono}.

Of particular interest is the case that \CatC{} is locally small and that there exists an object $G$ in \CatC{} (called a \emph{generator}) such that $\Hom_\CatC(G,\freeVar)\colon\CatC\to\CatSet$ is faithful;
see \cite[Corollary~4.5.9, p.155]{Bor94HandbookCat1}.
In that case, a morphism is mono if \emph{and only if} it is injective.
As noted above, the backward implication holds in every concrete category.
To show the forward implication, let $\varphi\colon X\to Y$ be a monomorphism.
To show that $\varphi$ is injective, let $x$ and $y$ be elements in the set underlying $X$ such that $\varphi(x)=\varphi(y)$.
Note that $x$ is an element of $\Hom_\CatC(G,X)$ and that $\varphi(x)$ is just the composition of morphisms $\varphi\circ x$ in $\CatC$.
Thus, the equality $\varphi(x)=\varphi(y)$ really means $\varphi\circ x=\varphi\circ y$.
Now it follows directly from the definition of monomorphism that $x=y$, as desired.

Dually, if a category has a cogenerator (an object $K$ such that $\Hom_\CatC(\freeVar,K)$ is faithful), then a morphism is epi if and only if it is surjective.

\begin{exa}
\label{exa:GenGp}
Let \CatGp{} denote the category of discrete groups and group homomorphisms, with the usual concretization that sends a group to its underlying set.
Then the group $\ZZ$ is a generator for \CatGp.
Indeed, given a group $G$, there is a natural bijection between elements in $G$ and group homomorphisms $\ZZ\to G$.
See also \cite[Example~4.5.17.c, p.160]{Bor94HandbookCat1}.
It follows that a morphism in \CatGp{} is mono if and only if it is injective.

The category \CatGp{} has no cogenerator;
see \cite[Proposition~4.7.3, p.169]{Bor94HandbookCat1}.
Nevertheless, a morphism in \CatGp{} is epi if and only if it is surjective.
The forward implication is not obvious;
see \cite{Lin70GpEpi}.
\end{exa}

\begin{exa}
\label{exa:Mon}
Let \CatMon{} denote the category of monoids and monoid homomorphisms.
The inclusion map $\varphi\colon\NN\to\ZZ$ is a non-surjective epimorphism.
To show that $\varphi$ is epi, let $M$ be a monoid, and let $\psi_1,\psi_2\colon\ZZ\to M$ be morphisms with $\psi_1\circ\varphi=\psi_2\circ\varphi$.
Then $\psi_1(k)=\psi_2(k)$ for all $k\geq 0$.
We have $1_M=\psi_1(0)=\psi_1(-k)\psi_1(k)$ and $1_M=\psi_2(k)\psi_2(-k)$ for all $k\geq 0$.
We deduce that
\[
\psi_1(-k)
=\psi_1(-k)\psi_2(k)\psi_2(-k)
=\psi_1(-k)\psi_1(k)\psi_2(-k)
=\psi_2(-k),
\]
for all $k\geq 0$.
Thus, $\psi_1=\psi_2$, as desired.
\end{exa}

\begin{exa}
\label{exa:nonInjMono}
Consider the category of pointed, path connected spaces with pointed, continuous maps.
We let $\TT=\{z\in\CC:|z|=1\}$ be the circle with base point $1$.
Let $\pi\colon(\RR,0)\to(S^1,1)$ be given by $\pi(t):=\exp(2\pi i t)$, for $t\in\RR$.
Then $\pi$ is a non-injective monomorphism.
To show that $\pi$ is a monomorphism, let $(X,x_0)$ be a path connected space, and let $f_1,f_2\colon(X,x_0)\to(\RR,0)$ be two pointed, continuous maps satisfying $\pi\circ f_1=\pi\circ f_2$.
Then $f_1(x_0)=0=f_2(x_0)$.
Given $x\in X$, choose a path from $x_0$ to $x$, that is, a continuous map $p\colon[0,1]\to X$ with $p(0)=x_0$ and $p(1)=x$.
Then $f_1\circ p$ and $f_2\circ p$ are two paths in $\RR$ starting at $0$.
We further have $\pi\circ f_1\circ p=\pi\circ f_2\circ p$.
Since $\pi$ is a covering, it has the unique path lifting property.
It follows that $f_1\circ p=f_2\circ p$ and hence
\[
f_1(x)=(f_1\circ p)(1)=(f_2\circ p)(1)=f_2(x).
\]
Thus, $f_1=f_2$, as desired.
\end{exa}

\begin{exa}
Let \CatCpctMetr{} be the category of compact, metric spaces and continuous mappings, with the usual concretization sending a topological space to its underlying set.
The one-point space is a generator for \CatCpctMetr.
The interval $[0,1]$ with its usual Hausdorff topology is a cogenerator for \CatCpctMetr;
see \cite[Proposition~4.7.8, p.173]{Bor94HandbookCat1}.
Thus, epimorphisms (monomorphisms) in \CatCpctMetr{} are precisely surjective (injective) continuous mappings.
\end{exa}

\section{Semiprojective and semiinjective objects}

The following definition is standard in category theory.

\begin{dfn}
\label{dfn:projInj}
Let \CatC{} be a category, and let $X$ be an object in \CatC.
Then $X$ is said to be \emph{projective} if for every epimorphism $\pi\colon Y\to Z$ and every morphism $\varphi\colon X\to Z$ there exists a morphism $\tilde{\varphi}\colon X\to Y$ such that $\pi\circ\tilde{\varphi}=\varphi$.
The morphism $\tilde\varphi$ is called a \emph{lift} of $\varphi$.

Dually, $X$ is said to be \emph{injective} if for every monomorphism $\iota\colon Z\to Y$ and every morphisms $\varphi\colon Z\to X$ there exists a morphism $\tilde{\varphi}\colon Y\to X$ such that $\tilde{\varphi}\circ\iota=\varphi$.
The morphism $\tilde\varphi$ is called an \emph{extension} of $\varphi$.

Thus, $X$ is projective (injective), if in the left (right) diagram below, for given solid arrows, the dashed arrow exists making the diagram commutative:
\[
\xymatrix{
& Y \ar@{->>}[d]^{\pi}
& & & Y \ar@{..>}[dl]_{\tilde{\varphi}} \\
X \ar[r]_{\varphi} \ar@{..>}[ur]^{\tilde{\varphi}} & Z
& & X & Z \ar[l]^{\varphi} \ar@{^{(}->}[u]_{\iota}
.
}
\]
\end{dfn}

\begin{exas}
\label{exa:projInjGp}
(1)
A group $G$ is projective (in \CatGp) if and only if $G$ is free.
The backward implication is easy to prove.
To show the forward implication, choose a free group $F$ and a surjective group homomorphism $\pi\colon F\to G$.
Using that $G$ is projective, we obtain a morphism $\tilde{\varphi}\colon G\to F$ that lifts the identity on $G$, that is, such that $\pi\circ\tilde{\varphi}=\id_G$.
This is shown in the following commutative diagram:
\[
\xymatrix{
& F \ar@{->>}[d]^{\pi} \\
G \ar[r]_{\id_G} \ar@{..>}[ur]^{\tilde{\varphi}} & G
.
}
\]
It follows that $\tilde{\varphi}$ is injective, and thus $G$ is (isomorphic to) a subgroup of $F$.
By the Nielsen-Schreier theorem, every subgroup of a free group is again free.
It follows that $G$ is free, as desired.

(2)
Eilenberg and Moore showed that the trivial group is the only injective object in \CatGp.
We include a short proof, which is a variation of the proof in \cite{Nog07EilMooThmNoInjGp}.

Let $G$ be an injective group, and let $g\in G$.
Let $F_2$ denote the free group of rank two, with generators $x$ and $y$.
Let $\varphi\colon F_2\to G$ be the morphism satisfying $\varphi(x)=1$ and $\varphi(y)=g$.
Let $\sigma\colon F_2\to F_2$ be the automorphism of $F_2$ satisfying $\sigma(x)=y$ and $\sigma(y)=x$.
We consider the semidirect product $F_2\rtimes_\sigma\ZZ_2$.
Let $\iota\colon F_2\to F_2\rtimes_\sigma\ZZ_2$ denote the natural inclusion morphism.
Use that $G$ is injective to obtain an extension $\tilde{\varphi}$ of $\varphi$.
This is shown in the following commutative diagram.
\[
\xymatrix{
G  \\
F_2 \ar[u]^{\varphi} \ar@{^{(}->}[r]_{\iota} & F_2\rtimes_\sigma\ZZ_2 \ar@{..>}[ul]_{\tilde{\varphi}}
.
}
\]
To simplify, we consider $F_2$ as a subgroup of $F_2\rtimes_\sigma\ZZ_2$.
Let $u\in F_2\rtimes_\sigma\ZZ_2$ be the element implementing $\sigma$.
Then $uxu^{-1}=y$.
We have $\tilde{\varphi}(x)=\varphi(x)=1$ and hence
\[
g
=\varphi(y)
=\tilde{\varphi}(uxu^{-1})
=\tilde{\varphi}(u)\tilde{\varphi}(x)\tilde{\varphi}(u)^{-1}
=\tilde{\varphi}(u)\tilde{\varphi}(u)^{-1}
= 1.
\]
Thus, $G=\{1\}$, as desired.
\end{exas}

Recall that a partially ordered set $I$ is said to be \emph{upward directed} (\emph{downward directed}) if for all $i,j\in I$ there exists $k\in I$ with $i,j\leq k$ (with $k\leq i,j$).
The following definition is standard.

\begin{dfn}
A \emph{direct system} (an \emph{inverse system}) in a category \CatC{} is an upward directed (downward directed) set $I$, together with objects $X_i$ for $i\in I$ and morphisms $\pi_{i,j}\colon X_i\to X_j$ for $i,j\in I$ with $i\leq j$, satisfying $\pi_{i,i}=\id_{X_i}$ for every $i\in I$ and satisfying $\pi_{i,k}=\pi_{j,k}\circ\pi_{i,j}$ for all $i,j,k\in I$ with $i\leq j\leq k$.
The morphisms $\pi_{i,j}$ are called the \emph{connecting morphisms} of the system.

Given a direct system $(I,X_i,\pi_{i,j})$, a \emph{direct limit} (also called \emph{inductive limit}) is an object $X$ together with a family $\pi=(\pi_{i,\infty})_{i\in I}$ of morphisms $\pi_{i,\infty}\colon X_i\to X$ satisfying $\pi_{j,\infty}\circ\pi_{i,j}=\pi_{i,\infty}$ for all $i,j\in I$ with $i\leq j$, and such that $(X,\pi)$ is universal with these properties.
Dually, given an an inverse system $(I,X_i,\pi_{i,j})$, an \emph{inverse limit} is an object $X$ together with a family $\pi=(\pi_{\infty,i})_{i\in I}$ of morphisms $\pi_{\infty,i}\colon X\to X_i$ satisfying $\pi_{i,j}\circ\pi_{\infty,i}=\pi_{\infty,j}$ for all $i,j\in I$ with $i\leq j$, and such that $(X,\pi)$ is universal with these properties.
\end{dfn}

The following definition of semiprojectivity was introduced by Blackadar in the sequential setting for the category of \ca{s};
see \cite[Definition~2.1]{Bla85NCShapeThy}.
See also \cite[Definition~2.10]{Bla85ShapeThy}.
The general (nonsequential) definition has also been considered in \cite{ChiLorThi18pre:nonsepSP}.

\begin{dfn}
\label{dfn:semiproj}
Let \CatC{} be a category, and let $X$ be an object in \CatC.
Then $X$ is said to be \emph{semiprojective} if for every inductive system $(I,Y_i,\pi_{i,j})$ in \CatC{} with connecting epimorphisms and for which the direct limit $\varinjlim Y_i$ exists, and for every morphism $\varphi\colon X\to\varinjlim Y_i$, there exist $i\in I$ and a morphism $\tilde{\varphi}\colon X\to Y_i$ such that $\pi_{i,\infty}\circ\tilde{\varphi}=\varphi$.
The morphism $\tilde\varphi$ is called a \emph{partial lift} of $\varphi$.

Dually, $X$ is said to be \emph{semiinjective} if for every inverse system $(I,Y_i,\iota_{i,j})$ in \CatC{} with connecting monomorphisms and for which the inverse limit $\varprojlim Y_i$ exists, and for every morphisms $\varphi\colon \varprojlim Y_i\to X$, there exist $i\in I$ and a morphism $\tilde{\varphi}\colon Y_i\to X$ such that $\tilde{\varphi}\circ\iota_{\infty,i}=\varphi$.
The morphism $\tilde\varphi$ is called an \emph{partial extension} of $\varphi$.
\end{dfn}

Given objects $X$ and $Y$ in a category, we say that $X$ is a \emph{retract} of $Y$ if there exist morphisms $\alpha\colon X\to Y$ and $\beta\colon Y\to X$ with $\beta\circ\alpha=\id_X$.
The proof of the following result is straightforward.

\begin{lma}
\label{prp:semiprojRetract}
Let \CatC{} be a category, and let $X$ and $Y$ be objects in \CatC.
Then: 
\begin{enumerate}
\item
If $X$ is a retract of $Y$, and if $Y$ is (semi)projective, then so is $X$.
\item
If $X$ and $Y$ are (semi)projective, then so is the coproduct $X\coprod Y$ (assuming it exists).
\end{enumerate}
\end{lma}

\begin{rmks}
(1)
Under suitable countability assumptions, it is usually enough to consider sequential direct limits in \autoref{dfn:semiproj}.
For example, a countable group $G$ is semiprojective if and only if for every sequential direct system $(\NN,D_k)$ of countable groups and with connecting epimorphisms, every morphism $G\to\varinjlim_k D_k$ has a partial lift.

To show the backward implication, let $(I,H_i,\pi_{i,j})$ be an arbitrary direct system with connecting epimorphisms in \CatGp, and let $\varphi\colon G\to\varinjlim_i H_i$ be a morphism.
Then there exist an increasing sequence of indices $i(0)\leq i(1)\leq\ldots$ in $I$, and countable subgroups $D_k\subseteq H_{i(k)}$ for all $k\in\NN$, such that the restriction of $\pi_{i(k),i(k+1)}$ to $D_k$ maps onto $D_{k+1}$, and such that $\varphi$ factors through $\varinjlim_k D_k$.
This means that there exists a morphism $\psi\colon G\to\varinjlim_k D_k$ such that $\varphi=\gamma\circ\psi$, where $\gamma\colon\varinjlim_k D_k\to\varinjlim_i H_i$ is the morphism obtained by the universal property of $\varinjlim_k D_k$ applied for the morphisms $(\pi_{i(k),\infty})_{|D_k}\colon D_k\to\varinjlim_i H_i$.
Given a partial lift $\tilde{\psi}\colon G\to D_k$ for $\psi$, we obtain a partial lift for $\varphi$ by composing $\tilde{\psi}$ with the inclusion $D_k\subseteq H_{i(k)}$.

The situation is shown in the following commutative diagram:
\[
\xymatrix@R-5pt@C-5pt{
& D_k \ar@{->>}[d] \ar@{^{(}->}[r]
& H_{i(k)} \ar@{->>}[d]^{\pi_{i(k),\infty}} \\
G \ar[r]^-{\psi} \ar@{..>}[ur]^{\tilde{\psi}} \ar@/_1pc/[rr]_{\varphi}
& \varinjlim_{k\in\NN} D_k \ar[r]^{\gamma}
& \varinjlim_{i\in I} H_i
.
}
\]

(2)
Similarly, a separable \ca{} $A$ is semiprojective if and only if every \stHom{} from $A$ to the direct limit of a sequential direct system of separable \ca{s} with surjective connecting maps has a partial lift.

(3)
The concept of direct and inverse limits in a category can be generalized to filtered (co)limits;
see \cite[Section~2.13, p.75ff]{Bor94HandbookCat1}.
In some categories, it may be appropriate to modify \autoref{dfn:semiproj} and consider filtered (co)limits instead of direct and inverse limits.
\end{rmks}

\begin{prp}
\label{prp:semiprojGp}
A group is semiprojective if and only if it is the retract of the free product of a finitely presented group and a free group.
\end{prp}
\begin{proof}
Let us show the backward implication.
Using \autoref{prp:semiprojRetract}, it remains to prove that every finitely presented group $H$ is semiprojective.
Choose a finitely generated free group $F$ and a finitely generated normal subgroup $N\lhd F$ such that $H$ is isomorphic to $F/N$.
We identify $H$ with $F/N$.
Let $r_1,\ldots,r_n\in N$ be a set of elements that generate $N$ as a normal subgroup of $F$.

To show that $H$ is semiprojective, let $(I,H_i,\pi_{i,j})$ be a direct system in \CatGp{} with connecting epimorphisms, and let $\varphi\colon H\to\varinjlim H_i$ be a morphism.
Let $\gamma\colon F\to H$ be the quotient map.
Since $F$ is free, and hence projective, we can choose $i_0\in I$ and a lift $\psi\colon F\to H_i$ of $\varphi\circ\gamma$.
Given $k\in\{1,\ldots,n\}$, we have
\[
\pi_{i,\infty}\big( \psi(r_k) \big)
= (\pi_{i,\infty}\circ\psi)(r_k)
= (\varphi\circ\gamma)(r_k)
= 1,
\]
which allows us to choose $i_k\geq i_0$ such that $\pi_{i,i_k}\big( \psi(r_k) \big)=1$.
Choose $i'\in I$ with $i_1,\ldots,i_n\leq i'$.
Then $\pi_{i,i'}$ maps $\psi(r_1),\ldots,\psi(r_n)$ to $1$.
It follows that $\psi(N)\subseteq\ker(\pi_{i,i'})$.
Thus, $\pi_{i,i'}\circ\gamma$ factors through $H$, which provides the desired partial lift.

To show the forward implication, assume that $G$ is a semiprojective group.
Choose a set $X$ and a surjective group homomorphism $\gamma\colon F(X)\to G$, where $F(X)$ denotes the free group on the set of generators $X$.
Set $N:=\ker(\gamma)$.
Given a subset $A\subseteq X$, we identify $F(A)$ in the obvious way with a subgroup of $F(X)$.
Set
\[
I := \big\{ (A,B) : A\subseteq X \text{ finite}, B\subseteq N \text{ finite}, B\subseteq F(A) \big\}.
\]
For $(A,B)\in I$ set $G_{(A,B)} := F(A)/\langle B\rangle$, where $\langle B\rangle$ denotes the normal subgroup of $F(A)$ generated by $B$.
Let $\gamma_{(A,B)}\colon F(A)\to G_{(A,B)}$ denote the quotient map.

We define a partial order on $I$ by setting $(A',B')\leq(A,B)$ if $A'\subseteq A$ and $B'\subseteq B$.
Then $I$ is an upward directed set.
For $(A,B)\in I$ we set
\[
I_{(A,B)}:= \big\{ (\tilde{A},\tilde{B})\in I : (\tilde{A},\tilde{B})\geq(A,B) \big\},
\]
and
\[
H_{(A,B)} := G_{(A,B)} \star F(X\times I_{(A,B)}).
\]

Given $(A',B')\leq(A,B)$, let us define a surjective morphism $\pi_{(A',B')}^{(A,B)}$ from $H_{(A',B')}$ to $H_{(A,B)}$.
The inclusion $F(A')\to F(A)$ induces a morphism $G_{(A',B')}\to G_{(A,B)}$.
We let $R$ be the subset of $X\times I_{(A',B')}$ such that $X\times I_{(A',B')}$ is the disjoint union of $A\times\{(A',B')\}$, $X\times I_{(A,B)}$, and $R$.
The first coordinate projection $A\times\{(A',B')\}\to A$ induces a morphism $F(A\times\{(A',B')\})\to F(A)$ that we postcompose with $\gamma_{(A,B)}$ to obtain a surjective morphism $F(A\times\{(A',B')\})\to G_{(A,B)}$.
We define $\pi_{(A',B')}^{(A,B)}$ as the free product of the morphisms $G_{(A',B')}\to G_{(A,B)}$, $F(A\times\{(A',B')\})\to G_{(A,B)}$, the identity on $F(X\times I_{(A,B)})$ and the trivial map $F(R)\to\{1\}$.
This is shown in the following diagram:
\[
\xymatrix@C-10pt{
H_{(A',B')}
\ar@{}[r]|{:=} \ar[d]_{\pi_{(A',B')}^{(A,B)}}
& G_{(A',B')} \ar@{}[r]|-{\star} \ar[d]
& F(A\times\{(A',B')\}) \ar@{}[r]|-{\star} \ar@{->>}[dl]
& F(X\times I_{(A,B)}) \ar@{}[r]|-{\star} \ar[dl]_{\cong}
& F(R) \\
H_{(A,B)} \ar@{}[r]|{:=}
& G_{(A,B)} \ar@{}[r]|-{\star}
& F(X\times I_{(A,B)})
}
\]
It is straightforward to verify that the maps $\pi_{(A',B')}^{(A,B)}$ are surjective and define an inductive system (over the index set $I$).
Moreover, there is a natural isomorphism $\varphi\colon G\to\varinjlim_{(A,B)\in I} G_{(A,B)}$.

Using that $G$ is semiprojective, we find $(A,B)\in I$ and a partial lift $\tilde{\varphi}\colon G\to G_{(A,B)}$ of $\varphi$.
This shows that $G$ is a retract of $G_{(A,B)}$, which is the free product of the finitely presented group $H_{(A,B)}$ and a free group.
\end{proof}

\begin{cor}
Every finitely presented group is semiprojective.
Moreover, every group is a direct limit of semiprojective groups
(and one may also assume that the connecting morphisms are surjective).
\end{cor}

\begin{prp}
\label{prp:semiinjGp}
The trivial group is the only semiinjective object of \CatGp.
\end{prp}
\begin{proof}
Let $G$ be a semiinjective group.
We show that $G$ is injective, whence it is trivial as noted in \autoref{exa:projInjGp}.

To show that $G$ is injective, let $H\subseteq K$ be an inclusion of groups, and let $\varphi\colon H\to G$ be a morphism.
We let $\star_{n\in\NN} K$ denote the free product of countably many copies of $K$, and for each $m\in\NN$ we let $\iota_m\colon K\to \star_{n\in\NN} K$ be the natural inclusion.
The amalgamated free product $\star_{n\in\NN, H} K$ is defined as the quotient of $\star_{n\in\NN} K$ by the normal subgroup generated by $\iota_n(h)\iota_m(h)^{-1}$, for $n,m\in\NN$ and $h\in H$.
For each $m\in\NN$ we let $\star_{n\geq m, H} K$ denote the subgroup of $\star_{n\in\NN, H} K$ generated by all except the first $m$ copies of $K$.
This defines a decreasing sequence of subgroups whose intersection is isomorphic to $H$.

Since $G$ is semiinjective, there exist $m$ and a partial extension $\tilde{\varphi}\colon \star_{n\geq m, H} K\to G$.
Composing with the morphism $\iota_m\colon K\to\star_{n\geq m, H} K$, we obtain a morphism $K\to G$ that extends $\varphi$, showing that $G$ is injective.
\end{proof}

The following definition is due to Borsuk.
For more details we refer to the books \cite{Bor67ThyRetracts} and \cite{Hu65ThyRetracts}.
Recall that a \emph{retract} from a topological space $Y$ to a subspace $X$ is a continuous map $r\colon Y\to X$ that satisfies $r(x)=x$ for all $x\in X$.

\begin{dfn}
\label{dfn:ANR}
Let $X$ be a metric space.
Then:
\begin{enumerate}
\item
$X$ is called an \emph{absolute retract} if whenever $X$ is embedded as a closed subset of another metric space $Y$, there exists a retract $Y\to X$.
\item
$X$ is called an \emph{absolute neighborhood retract} if whenever $X$ is embedded as a closed subset of another metric space $Y$, there exist a neighborhood $U$ of $X$ in $Y$ and a retract $U\to X$.
\end{enumerate}
\end{dfn}

The equivalence between~(1) and~(2) in the following result is a standard fact about absolute neighborhood retracts;
see for example \cite[Theorem~III.3.1, III.3.2, p.83f]{Hu65ThyRetracts}.
The equivalence between~(2) and~(3) follows using that compact, metric spaces are normal.

\begin{prp}
Let $X$ be a compact, metric space.
Then the following are equivalent:
\begin{enumerate}
\item
$X$ is an absolute (neighborhood) retract.
\item
Given a compact, metric space $Y$ and a closed subset $Z\subseteq Y$, and given a continuous map $\varphi\colon Y\to X$, there exists an extension of $\varphi$ to a continuous map $\tilde{\varphi}\colon Y\to X$ (there exists a closed neighborhood $C$ of $Z$ in $Y$ and an extension $\tilde{\varphi}\colon C\to X$).
\item
$X$ is a (semi)injective object in the category \CatCpctMetr.
\end{enumerate}
\end{prp}

\section{\texorpdfstring{$C^*$-algebras}{C*-algebras}}

A \ca{} is a Banach algebra $A$ with an involution such that $\|a^*a\|=\|a\|^2$ for all $a\in A$.
A \stHom{} between \ca{s} is a multiplicative, ${}^*$-preserving, linear map.
We let \CatCa{} denote the category of \ca{s} and \stHom{s}.
The naive concretization of \CatCaUnit{} associates to every \ca{} its (usual) underlying set.
However, this functor $\CatCa\to\CatSet$ is not representable.

Nevertheless, \CatCa{} has a generator.
Indeed, let $G:=C^*(x:\|x\|\leq 1)$ be the universal \ca{} generated by a contraction.
Given a \ca{} $A$, there is a natural bijection between $\Hom_{\CatCa}(G,A)$ and the elements in the unit ball of $A$.
To see that $G$ is a generator, we note that two \stHom{s} $A\to B$ are equal if and only if they agree on the unit ball of $A$.
Hence, a morphism in \CatCa{} is mono if and only if it is injective.

The category \CatCa{} has no cogenerator.
(The proof is analogous to that for \CatGp.)
Nevertheless, a morphism in \CatCa{} is epi if and only if it is surjective.
As for \CatGp, the forward implication is not obvious;
see \cite[Proposition~2]{Rei70EpiSurj} and \cite{HofNee95EpiCaSurj}.

It follows that isomorphisms in \CatCa{} are exactly the bijective \stHom{}, also called ${}^*$-isomorphisms, and such maps are automatically isometric.

For simplicity, we will restrict attention to the subcategory \CatCaSepUnit{} of unital, separable \ca{s} and unital \stHom{s}.
We let $\CatCaCommSepUnit$ denote the full subcategory of \CatCaSepUnit{} of \emph{abelian}, unital, separable \ca{s}.

\begin{exas}
\label{exa:ca}
(1)
Given a Hilbert space $H$, the algebra $\mathcal{B}(H)$ of bounded linear operators on $H$, equipped with the operator norm and the natural involution, is a unital \ca{}.
By the Gelfand-Naimark theorem, every \ca{} is ${}^*$-isomorphic to norm-closed ${}^*$-subalgebra of $\mathcal{B}(H)$ for some Hilbert space $H$;
see \cite[Corollary~II.6.4.10, p.109]{Bla06OpAlgs}.

(2)
For the Hilbert space $H=\ell^2(\{1,2,\ldots,n\})$, we obtain that $\mathcal{B}(H)\cong M_n(\CC)$, the algebra of complex $n\times n$-matrices, has the structure of a \ca{}.
By the Artin-Weddenburn theorem, a \ca{} is finite-dimensional (as a complex vector space) if and only if it is isomorphic to a finite direct sum of matrix algebras.

(3)
Let $X$ be a compact, metric space.
Set
\[
C(X) := \big\{ f\colon X\to\CC : f \text{ is continuous} \big\},
\]
equipped with pointwise addition, multiplication and involution, and with the norm
\[
\|f\| := \sup \big\{ |f(x)| : x\in X \big\},
\]
for $f\in C(X)$.
Then $C(X)$ is a unital, commutative, separable \ca{}.

If $Y$ is another compact, metric space, and if $\varphi\colon X\to Y$ is a continuous map, then $\varphi^*\colon C(Y)\to C(X)$ given by $\varphi^*(f):=f\circ\varphi$ is a unital \stHom.
This defines a contravariant functor $C(\freeVar)\colon\CatCpctMetr\to\CatCaCommSepUnit$.
\end{exas}

\begin{prp}[{Gelfand; \cite[Theorem~II.2.2.6, p.61]{Bla06OpAlgs}}]
\label{prp:equivCHAbCa}
Every abelian, unital, separable \ca{} is isomorphic to $C(X)$ for some compact, metric space $X$.
Moreover, the functor $C(\freeVar)\colon\CatCpctMetr\to\CatCaCommSepUnit$ defines a (contravariant) equivalence of categories.
\end{prp}

\begin{rmk}
Projectivity and semiprojectivity of \ca{s} is defined with respect to the category \CatCa{} of all \ca{s}.
When considering the category \CatCaUnit, the notion of projectivity changes.
On the other hand, a unital \ca{} is semiprojective (in \CatCa) if and only if it is semiprojective in \CatCaUnit.
\end{rmk}

\begin{rmk}
\label{rmk:semiprojImplANR}
Let $X$ be a compact, metric space.
If $C(X)$ is (semi)projective in \CatCaUnit{}, then it is also (semi)projective in \CatCaCommSepUnit{} - since there are fewer lifting problems to solve.
Using the (contravariant) equivalence between \CatCaCommSepUnit{} and \CatCpctMetr{} from \autoref{prp:equivCHAbCa}, we deduce that $C(X)$ is (semi)projective in \CatCaCommSepUnit{} if and only if $X$ is (semi)injective in \CatCpctMetr{}.
Thus, if $C(X)$ is (semi)projective in \CatCaUnit, then $X$ is an absolute (neighborhood) retract.
\end{rmk}

\begin{exas}
\label{exa:semiproj}
(1)
$\CC=C(\mathrm{pt})$ and $C([0,1])$ are projective in \CatCaUnit.

(2)
Set $S^1:=\{z\in\CC : |z|=1 \}$, the unit circle.
The \ca{} $C(S^1)$ is semiprojective, but not projective in \CatCaUnit.

(2)
Set $D^2:=\{z\in\CC : |z|\leq 1 \}$, the two-disc.
Then $D^2$ is an absolute retract.
However, the \ca{} $C(D^2)$ is not even semiprojective.

(3)
If $X$ is a compact, metric space with $S^1\subseteq X$, then $C(X)$ is not projective in \CatCaUnit.
Indeed, assume that $\iota\colon S^1\hookrightarrow X$ is an embedding.
Let $\varphi\colon S^1\to D^2$ be the inclusion map.
Since $D^2$ is an absolute retract, there exists an extension of $\varphi$ to a continuous map $\tilde{\varphi}\colon X\to D^2$.
This is shown in the left commutative diagram below.
\[
\xymatrix@R-5pt{
D^2
& & & & & \mathcal{T} \ar[d]^{\pi} \\
S^1 \ar[u]^{\varphi} \ar@{^{(}->}[r]_{\iota}
& X \ar@{-->}[ul]_{\tilde{f}}
& & C(D^2) \ar[dr]^{\tilde{\varphi}^*} \ar[rr]^{\varphi^*}
& & C(S^1) \\
& & & & C(X) \ar[ur]^{\iota^*}
.}
\]
We let $\mathcal{T}$ denote the Toeplitz algebra, that is, the sub-\ca{} of $\mathcal{B}(\ell^2(\NN))$ generated by the unilateral shift on $\ell^2(\NN)$.
Sending this shift to the identity map $\id\in C(S^1)$ induces a surjective \stHom{} $\pi\colon\mathcal{T}\to C(S^1)$.
The situation is shown in the right commutative diagram above.

If $C(X)$ were projective in \CatCaUnit{}, then there would exist a lift $\psi\colon C(X)\to\mathcal{T}$ for $\iota^*$.
Then $\psi\circ\tilde{\varphi}^*$ is a lift for $\varphi^*$.
The image of the identity map $\id\in C(D^2)$ under $\psi\circ\tilde{\varphi}^*$ is a normal element in $\mathcal{T}$ that lifts the unitary $\id\in C(S^1)$.
Using the Fredholm index one can show that no such lift exists, showing that $C(X)$ is not projective in \CatCaUnit.

(4)
If $X$ is a compact, metric space with $D^2\subseteq X$, then $C(X)$ is not semiprojective in \CatCaUnit.
This is shown similarly as in~(3);
see \cite[Remark~3.3]{SoeThi12CharCommutSP}.
\end{exas}

Let $X$ be a compact, metric space such that $C(X)$ is (semi)projective in \CatCaUnit.
As observed in \autoref{rmk:semiprojImplANR}, it follows that $X$ is an absolute (neighborhood) retract.
The above examples show that the converse does not hold.

\begin{cnj}[{Blackadar, \cite[II.8.3.8, p.163]{Bla06OpAlgs}}]
\label{cnj:Blackadar}
Let $X$ be a compact, metric space.
Then $C(X)$ is semiprojective if and only if $X$ is an absolute neighborhood retract with $\dim(X)\leq 1$.
\end{cnj}

The analog of this conjecture for projectivity was solved by Chigogidze and Dranishnikov:

\begin{thm}[{\cite[Theorem~4.3]{ChiDra10NcAR}}]
\label{prp:CharNcAR}
Let $X$ be a compact, metric space.
Then $C(X)$ is projective in \CatCaUnit{} if and only if $X$ is an absolute retract with $\dim(X)\leq 1$.
\end{thm}

Let us sketch the proof of the forward implication of \autoref{prp:CharNcAR}.
Assume that $C(X)$ is projective in \CatCaUnit{}.
Then $X$ is an absolute retract;
see \autoref{rmk:semiprojImplANR}.
To show that $\dim(X)\leq 1$, assume that $\dim(X)\geq 2$.
By \cite[Proposition~3.2]{ChiDra10NcAR}, it follows that $S^1\subseteq X$.
As sketched in \autoref{exa:semiproj}(3), this implies that $C(X)$ is not projective.

We remark that the topological result that $S^1$ embeds into $X$ uses both that $X$ is an absolute (neighborhood) retract and that $\dim(X)\geq 2$.
Indeed, the space $[0,1]$ is an example of an absolute retract that does not admit an embedding of the circle.
On the other hand, there exist compact metric spaces with $\dim(X)=\infty$ such that every closed subset of $X$ satisfies either $\dim(X)=0$ or $\dim(X)=\infty$.
In particular, such a space does not admit an embedding of the circle.
The point is that such a behaviour is not possible for `well-behaved' spaces such as absolute (neighborhood) retracts.

Together with Adam S{\o}rensen, we confirmed Blackadar's conjecture.

\begin{thm}[{\cite[Theorem~1.2]{SoeThi12CharCommutSP}}]
\label{prp:CharNcANR}
Let $X$ be a compact, metric space.
Then $C(X)$ is semiprojective if and only if $X$ is an absolute neighborhood retract with $\dim(X)\leq 1$.
\end{thm}

Let us sketch the proof of the forward implication of \autoref{prp:CharNcANR}.
Assume that $C(X)$ is semiprojective in \CatCaUnit{}.
Then $X$ is an absolute neighborhood retract;
see \autoref{rmk:semiprojImplANR}.
To show that $\dim(X)\leq 1$, assume that $\dim(X)\geq 2$.
If we could deduce that $D^2$ embeds into $X$, then we would conclude that $C(X)$ is not semiprojective as mentioned in \autoref{exa:semiproj}(3).

The problem is that $\dim(X)\geq 2$ does not imply $D^2\subseteq X$.
Indeed, Bing and Borsuk constructed an absolute retract $Y$ such that $\dim(Y)=3$, but such that $D^2$ does not embed into $Y$;
see \cite{BinBor64ANR3dimNoDisk}.
Thus, we cannot assume that a disc embeds into $X$.
Instead, we use the following topological result:

\begin{lma}[{\cite[Remark~3.4]{SoeThi12CharCommutSP}}]
\label{prp:highDimANR}
Let $X$ be a compact, metric space that is an absolute neighborhood retract.
Assume that $\dim(X)\geq 2$.
Then $X$ contains either the space $C_1$ of disjoint `smaller and smaller circles', the Hawaiian earrings space $C_2$, or the space $C_3$ that is a variant of the Hawaiian earrings, each given as subsets of $\RR^2$: ($S(x,r)$ denotes the circle of radius $r$ around $x$.)
\begin{align*}
C_1 &= \big\{ (0,0) \big\} \cup \bigcup_{k\geq 1}S((1/2^k,0),1/(4\cdot 2^k)), \quad
C_2 =\bigcup_{k\geq 1}S((1/2^k,0),1/2^k), \\
C_3 &= \big\{ (x,x),(x,-x) : x\in[0,1] \big\} \cup \bigcup_{k\geq 1} \{1/k\}\times[-1/k,1/k].
\end{align*}
\begin{figure}[ht] \centering
\subfigure[Space $C_1$]{
\psset{unit=0.0117cm}
\degrees
\begin{pspicture}(-70,350)(70,-20)
\psline[linestyle=dotted, linewidth=2pt, dotsep=2pt](0,-10)(0,15)
\pscircle(0, 282){54}
\pscircle(0, 156){36}
\pscircle(0, 72){24}
\end{pspicture}
}
\quad
\subfigure[Space $C_2$]{
\psset{unit=0.018cm}
\begin{pspicture}(-20,110)(180,-110)
\pscircle(81, 0){81}
\pscircle(54, 0){54}
\pscircle(36, 0){36}
\pscircle(24, 0){24}
\pscircle(16, 0){16}
\pscircle[linestyle=dotted, linewidth=2pt, dotsep=1.2pt](10, 0){10}
\end{pspicture}
}
\quad
\subfigure[Space $C_3$]{
\psset{unit=1.17cm}
\degrees
\begin{pspicture}(-0.2,1.7)(3.4,-1.7)

\rput{-90}(2,2){
\parabola(0.4,0.56)(2,-2)
}

\psarc(0,0){2.981}{-32.192}{32.192}
\psarc(0,0){1.988}{-38.776}{38.776}
\psarc(0,0){1.325}{-46.253}{46.253}
\psarc(0,0){0.883}{-54.334}{54.334}
\psarc(0,0){0.589}{-62.431}{62.431}
\psarc[linestyle=dotted, linewidth=2pt, dotsep=1.5pt](0,0){0.393}{-69.775}{69.775}
\end{pspicture}
}
\end{figure}
\end{lma}

A boosted version of the argument in \autoref{exa:semiproj}(3) shows the following result, which together with \autoref{prp:highDimANR} shows the forward implication of \autoref{prp:CharNcANR}.

\begin{lma}[{\cite{SoeThi12CharCommutSP}}]
Let $X$ be a compact, metric space.
If $X$ contains any of the spaces $C_1$, $C_2$, or $C_3$ from \autoref{prp:highDimANR}, then $C(X)$ is not semiprojective.
\end{lma}


\providecommand{\bysame}{\leavevmode\hbox to3em{\hrulefill}\thinspace}
\providecommand{\noopsort}[1]{}
\providecommand{\mr}[1]{\href{http://www.ams.org/mathscinet-getitem?mr=#1}{MR~#1}}
\providecommand{\zbl}[1]{\href{http://www.zentralblatt-math.org/zmath/en/search/?q=an:#1}{Zbl~#1}}
\providecommand{\jfm}[1]{\href{http://www.emis.de/cgi-bin/JFM-item?#1}{JFM~#1}}
\providecommand{\arxiv}[1]{\href{http://www.arxiv.org/abs/#1}{arXiv~#1}}
\providecommand{\doi}[1]{\url{http://dx.doi.org/#1}}
\providecommand{\MR}{\relax\ifhmode\unskip\space\fi MR }
\providecommand{\MRhref}[2]{%
  \href{http://www.ams.org/mathscinet-getitem?mr=#1}{#2}
}
\providecommand{\href}[2]{#2}

\end{document}